\documentclass[a4paper, 12pt]{amsart}

\usepackage{ucs}
\usepackage[utf8x]{inputenc}
\usepackage[ngerman, english]{babel}
\usepackage{amsfonts,amsmath, amssymb}
\usepackage{fancyhdr}
\usepackage[fixlanguage]{babelbib}
\usepackage{url}
\usepackage[arrow, matrix, curve]{xy}
\usepackage{hyperref}
\usepackage{color}
\usepackage{mathtools, tensor}
\usepackage{cite}
\usepackage{tikz-cd}

\newtheorem{thm}{Theorem}[section]
\newtheorem{lem}[thm]{Lemma}

\newtheorem{prop}[thm]{Proposition}

\newtheorem*{thm1}{Theorem 4.4}
\newtheorem*{thm2}{Theorem 5.4}
\newtheorem*{prop1}{Proposition 6.1}

\theoremstyle{definition}
\newtheorem{defi}[thm]{Definition}

\theoremstyle{remark}
\newtheorem{rem}[thm]{Remark}

\newcommand{\spinc}{$\text{spin}^c$ }
\newcommand{\Spinc}{\mathrm{Spin}^c}
\newcommand{\spincbundle}{\Sigma^c M}
\newcommand{\spincibundle}[1]{\Sigma^c_{#1} M}

\newcommand{\fried}{\nabla^{\mu}}
\newcommand{\friedi}{\nabla^{\mu , A_j}}
\newcommand{\friedcomcon}{\prescript{A_j}{}{\nabla}^{\mu , A_g}}
\newcommand{\conng}{\nabla^g}
\newcommand{\connh}{\nabla^h}
\newcommand{\comcon}[3]{\tensor*[^{#1}_{}]{\nabla}{^{#2}_{#3}}}
\newcommand{\conlaplace}{\nabla^{\ast} \nabla}
\newcommand{\spinccon}{\nabla^A}
\newcommand{\spincicon}[1]{\nabla^{A_{#1}}}
\newcommand{\spinclaplace}{\left. \nabla^A \right.^{\ast} \nabla^A}
\newcommand{\friedspinclaplace}{\left. \nabla^{\mu} \right.^{\ast} \nabla^{\mu}}
\newcommand{\friedspincilaplace}{\left. \nabla^{\mu , A_j} \right.^{\ast} \nabla^{\mu , A_j}}
\newcommand{\Diraci}{D^{A_j}}

\newcommand{\Norm}[2]{\Vert #1 \Vert_{#2}}
\newcommand{\Cka}[1]{$C^{#1, \alpha}$}
\newcommand{\Lspinc}{L^2(\Sigma^c M)}
\newcommand{\Lspinci}{L^2(\Sigma^c_j M)}
\newcommand{\sob}{W^{1,2}(\Sigma^c M)}
\newcommand{\friedsob}{W^{1,2}_{\mu}(\Sigma^c M)}
\newcommand{\hoelderspincu}{C^1_c (\Sigma^c M_{\vert U})}

\newcommand{\summeab}{\sum_{a,b=1}^n}
\newcommand{\summea}{\sum_{a=1}^n}
\newcommand{\summei}{\sum_{i=1}^n}

\newcommand{\sphere}{\mathbb{S}}
\newcommand{\bbN}{\mathbb{N}}
\newcommand{\bbR}{\mathbb{R}}
\newcommand{\bbC}{\mathbb{C}}
\newcommand{\bbZ}{\mathbb{Z}}
\newcommand{\bbT}{\mathbb{T}}
\newcommand{\eps}{\varepsilon}

\newcommand{\onframe}{(e_1 , \ldots , e_n)}
\newcommand{\coordframe}{(\partial_1 , \ldots , \partial_n)}

\newcommand{\sequence}[1]{$\left( #1_j \right)_{j \in \mathbb{N}}$}
\newcommand{\dvol}[1]{\mathrm{dvol}_{#1}}
\newcommand{\Iplus}{I_{\eps ,\mu}^+}
\newcommand{\Iminus}{I_{\eps , \mu}^-}

\newcommand{\mfdspace}{\mathcal{M}(n, \Lambda , i_0 , d)}

\newcommand{\precompRic}{\mathcal{M}^{\mathbb{S}^1}(n,\Lambda, i_0, d, K)}
\newcommand{\kasten}{\hfill $\square$ }

\DeclareMathOperator{\Ric}{Ric}
\DeclareMathOperator{\inj}{inj}
\DeclareMathOperator{\diam}{diam}
\DeclareMathOperator{\Id}{Id}
\DeclareMathOperator{\Scal}{Scal}
\DeclareMathOperator{\R}{R}
\DeclareMathOperator{\Gl}{GL}
\DeclareMathOperator{\So}{SO}
\DeclareMathOperator{\Spin}{Spin}
\DeclareMathOperator{\Image}{Im}


\begin{document}

\nocite{*}

\title{Eigenvalue pinching on $\text{SPIN}^C$ manifolds}

\author{Saskia Roos}
\address{Max-Planck-Institut für Mathematik, Bonn, Germany}
\email{saroos@mpim-bonn.mpg.de}

\begin{abstract}
We derive various pinching results for small Dirac eigenvalues using the classification of \spinc and spin manifolds admitting nontrivial Killing spinors. For this, we introduce a notion of convergence for \spinc manifolds which involves a general study on convergence of Riemannian manifolds with a principal $\sphere^1$-bundle. We also analyze the relation between the regularity of the Riemannian metric and the regularity of the curvature of the associated principal $\sphere^1$-bundle on \spinc manifolds with Killing spinors.
\end{abstract}

\maketitle

\section{Introduction}

Eigenvalue pinching on closed manifolds is an important and widely studied topic in Riemannian geometry. It gives insight into the relation between the spectrum of an operator and the topology of the manifold. One of these studied operators is the Dirac operator on spin and \spinc manifolds. For example, Ammann and Sprouse have shown in \cite[Theorem 1.8]{Ammann} that a spin manifold $M$ with $r(n)$ Dirac eigenvalues close to the Friedrich bound and an appropriate lower bound on the scalar curvature implies that $M$ is diffeomorphic to a manifold of constant curvature. Here $r(n) = 1$ if $n=2,3$ and $r(n) = \exp (\log(2) (\left[\frac{n}{2}\right] -1 ))+1$ if $n>3$. The limit case of the Friedrich inequality only contains spin manifolds with real Killing spinors whose geometry was described by B\"ar \cite{Baer} after a series of partial result of Friedrich, Grunewald, Kath and Hijazi (cf.\ \cite{Friedrich81} \cite{FriedrichGrunewald} \cite{FriedrichKath89} \cite{FriedrichKath90} \cite{FriedrichKath99} \cite{Hijazi}). 
Hence, Ammann and Sprouse conjectured that \cite[Theorem 1.8]{Ammann} should also be valid with a lower value $\tilde{r}(n) < r(n)$.

This problem was considered by Vargas \cite{Vargas}. He introduced the concept of almost Killing spinor sequences which describes a sequence of spinors together with a sequence of metrics on a spin manifold converging to a nontrivial Killing spinor. Studying the convergence of this sequence and combining it with Gromov's compactness theorem for manifolds he derived an improved version of \cite[Theorem 1.8]{Ammann} for simply-connected spin manifolds, \cite[Theorem 5.4.1]{Vargas}.
\newline

In this paper we define almost Killing spinors sequences on the larger class of \spinc manifolds. These will be used to derive pinching results on \spinc and spin manifolds with small spinorial Laplace eigenvalues or Dirac eigenvalues close to the Friedrich bound.

After recalling the basic definitions and properties of \spinc manifolds and Killing spinors, we shortly explain how to identify spinors of different metric \spinc structures following \cite{IdentifySpin} in Section 3.

In Section 4 we will define almost Killing spinor sequences on \spinc manifolds. One of the main points we need to deal with is to derive an applicable notion of convergence of \spinc manifolds. As the \spinc structure depends on an associated principal $\sphere^1$-bundle we first study the convergence of principal $\sphere^1$-bundles with connection over closed Riemannian manifolds. This leads to one of the main results of this paper.
\begin{thm1}
Let \sequence{P_j , A} be a sequence of principal $\sphere^1$-bundles with connection over a fixed compact Riemannian manifold $(M,g)$. For each $j$ let $\Omega_j$ be the 2-form representing the curvature of $A_j$. If there is a non-negative $K$ such that $\Vert \Omega_j \Vert_{C^{k, \alpha}} \leq K$ for all $j$, then for any $\beta <\alpha$ there are a principal $\sphere^1$-bundle $P$ with a $C^{k+1, \beta}$-connection $A$ and a subsequence, again denoted by \sequence{P_j , A} together with principal bundle isomorphisms $\Phi_j : P \rightarrow P_j$ such that $\Phi_j^{\ast} A_j$ converges to $A$ in the $C^{k+1 , \beta}$-norm
\end{thm1}
Afterwards we define almost Killing spinor sequences on \spinc manifolds and study their convergence behavior.

In Section 5 we analyze the regularity of \spinc manifolds with Killing spinors as \spinc manifolds with a Killing spinor are, in contrast to the spin case, in general not Einstein. We show that the existence of a nontrivial Killing spinor leads to an equation for the Ricci curvature of the manifold. Using harmonic coordinates we conclude with the results of \cite{DeTurckKazdan}:
\begin{thm2}
Let $(M,g)$ be a Riemannian \spinc manifold with a \Cka{1}-metric $g$,  \Cka{l}-curvature form $\Omega$ on the associated principal $\sphere^1$-bundle $P$, $l \geq 0$, and a nontrivial Killing spinor $\varphi$. Then $g$ is \Cka{l+2} in harmonic coordinates.
\end{thm2}

Outgoing from Theorem \ref{fixedConvergence} we define the space $\precompRic$, see Definition \ref{precompRic}, and prove in Section 6 
\begin{prop1}
Let $\Lambda$, $i_0$, $d$, $K$ and $k$ be given positive real numbers, $\mu$ a given real number and $n$ a given natural number. Let $(M,g)$ be a \spinc manifold in $\precompRic$. For every $\delta > 0$ there exists a positive $\eps = \eps(n, \Lambda, i_0, d, K,  k, \mu, \delta) > 0$ such that $\lambda_k (\friedspinclaplace)  < \eps$ implies that $(M,g)$ has \Cka{1}-distance smaller than $\delta$ to a \spinc manifold with $k$ linearly independent Killing spinors with Killing number $\mu$. Furthermore, $g$ is at least \Cka{2} in harmonic coordinates.
\end{prop1}
This proposition is the basis for all pinching results in this section. For $\mu = 0$ we combine this proposition with the geometric description of \spinc and spin manifolds with parallel spinors obtained in \cite{WangSimply},\cite{WangNonSimply} and \cite{Moroianu}.

Using the Schr\"odinger-Lichnerowicz formula we prove a similar result to Proposition \ref{KillFriedrich} for Dirac eigenvalues which leads again to eigenvalue pinching results for Dirac eigenvalues close to the Friedrich bound. For example, we show that even resp.\ odd dimensional simply-connected \spinc manifolds with one resp.\ two Dirac eigenvalues close to the Friedrich bound are already spin. Combining this with the geometric description of spin manifolds with real Killing spinors in \cite{Baer} we show that simply-connected \spinc manifolds with a specified number of Dirac eigenvalues close to the Friedrich bound are already diffeomorphic to the sphere.

As an application of our results, we show in Section 7 that using \cite[Theorem 3.1]{DaiWangWei}, the absolute value of the Killing number of a real Killing spinor is bounded from below by a positive constant in the class of $n$-dimensional Riemannian manifolds with bounded Ricci-curvature and diameter and with injectivity radius bounded from below by a positive constant. 

\subsection*{Acknowledgment}
It is a great pleasure for me to thank my PhD supervisors Werner Ballmann and Bernd Ammann for suggesting this problem and many helpful mathematical discussions. Furthermore I want to thank Bernd Ammann and the SFB 1085 for inviting me to Regensburg. I would also like to thank Christian Blohmann, Asma Hassannezhad, Henrik Matthiesen and Fabian Spiegel for interesting mathematical discussions and their advisory support. Moreover, I am very grateful for the support and hospitality of the Max-Planck-Intitut for Mathematics in Bonn.

\section{$\text{Spin}^c$ manifolds and Killing spinors}

For the reader's convenience we first collect some well-known facts about \spinc manifolds. For more detail see \cite{LawsonMichelsohn}, \cite{Friedrich} and \cite{SpinorialApproach}.

\begin{defi}[\spinc structure]
Let $\xi: \widetilde{\Gl}(n) \rightarrow \Gl(n)$ denote the nontrivial two-fold covering of $\Gl(n)$ and set 
\begin{align*}
\widetilde{\Gl}^c(n) = \widetilde{\Gl}(n) \times_{\bbZ_2} \sphere^1 & \rightarrow \Gl(n) \times \sphere^1\\
[A, u ] &\mapsto (\xi(A), u^2).
\end{align*}
A manifold $M$ with frame bundle $P_{\Gl}M$ admits a \textit{topological \spinc structure} if there is a principal $\sphere^1$-bundle $P$ such that there exists a principal $\widetilde{\Gl}^c$-bundle $P_{\widetilde{\Gl}^c} M$ that is a two-fold covering of $P_{\Gl}M \times P$ compatible with the associated two-fold group covering.

On a Riemannian manifold $(M,g)$ a \textit{metric \spinc structure} is the preimage of $P_{SO}M \times P$ of the topological \spinc structure, where $P_{\So}M$ consists of positive oriented orthonormal frames of $TM$. This preimage defines a principal $\Spinc$-bundle $P_{\Spinc}M$ with 
\begin{align*}
\Spinc(n) \coloneqq \Spin(n) \times_{\bbZ_2} \sphere^1 \subset \widetilde{\Gl}^c(n).
\end{align*}
\end{defi}
\vspace*{5pt}
Since any metric \spinc structure extends uniquely to a topological \spinc structure, they have the same equivalence classes.

We now introduce \textit{spinors} on a \spinc manifold as sections of the \spinc bundle $\spincbundle = P_{\Spinc}M \times_{\delta} \Sigma_n$, where $\delta: \Spinc(n) \to \Gl( \Sigma_n)$ denotes the canonical complex \spinc representation on the complex vector space $\Sigma_n$. The \spinc bundle $\spincbundle$ is endowed with a natural Hermitian inner product.

A connection $\spinccon$ on $\spincbundle$ is determined by a lift of the Levi-Civita connection of $(M,g)$ together with a connection 1-form $A$ on $P$ to $\spincbundle$. The Hermitian inner product and Clifford multiplication is parallel with respect to $\spinccon$.

The \textit{spinorial Laplacian} is defined as $\spinclaplace$ where $\left. \nabla^A \right.^{\ast}$ is the $L^2$-adjoint of $\nabla^A$. On the other hand, the \textit{Dirac operator} is defined by its action on spinors $\varphi$, given by $D^A \varphi = \summei e_i \cdot \spinccon_{e_i} \varphi$ in any orthonormal frame $\onframe$. These two operators are closely related by the Schr\"odinger-Lichnerowicz formula.
\begin{equation}\label{SLformula}
(D^A)^2\varphi = \spinclaplace \varphi + \frac{1}{4} \Scal \, \varphi + \frac{i}{2} \Omega \cdot \varphi
\end{equation}
for all spinors $\varphi$, where $\Scal$ denotes the scalar curvature of the manifold. Here $\Omega = -i F_A \in \Gamma(\Lambda^2 T^{\ast}M)$, where $F_A$ is the curvature of the connection 1-form $A$. The action of a $k$-form $\omega$ on a spinor $\varphi$ is defined by
\begin{align*}
\omega \cdot \varphi \coloneqq \sum_{i_1 < \ldots < i_k} \omega(e_{i_1}, \ldots , e_{i_k}) e_{i_1}\cdot \ldots \cdot e_{i_k} \cdot \varphi,
\end{align*}
in any orthonormal frame $\onframe$.

At this point, we remark that any spin manifold is \spinc. To see this, take the trivial principal $\sphere^1$-bundle $P$ and extend the spin bundle via the inclusion $\Spin(n) \hookrightarrow \Spin^c(n)$. By choosing $A=0$ the spinorial connection extends canonically to the \spinc bundle $\spincbundle$. Thus, all results stated for \spinc manifolds are also valid for spin manifolds.

As mentioned in the introduction we are mainly interested in Killing spinors. A spinor $\varphi$ is called a \textit{Killing spinor} if there exists a complex number $\mu$ such that $\spinccon_X \varphi = \mu X \cdot \varphi $ for any vector field $X$. The number $\mu$ is called Killing number. Obviously, $D^A \varphi = -  \mu n \varphi$. 

Starting from the identity
\begin{equation}\label{RicSpin}
\summei e_i \cdot \R^A(X,e_i)\varphi = - \frac{1}{2}\Ric(X)\cdot \varphi + \frac{i}{2} (X \lrcorner \Omega) \cdot \varphi
\end{equation}
which holds for any vector field $X$ and spinor $\varphi$ in any orthonormal frame, we obtain the following relation for Killing spinors.

\begin{lem}\label{KillRic}
Let $\varphi$ be a Killing spinor with Killing number $\mu$. Then
\begin{align*}
\Ric(X,Y)\vert \varphi \vert^2 &= 4 \mu^2 (n-1) g(X,Y) \vert \varphi \vert^2 \\
& \ \ \  + \frac{i}{2} \langle ( (X \lrcorner \Omega) \cdot Y) - Y \cdot (X \lrcorner \Omega))\cdot \varphi, \varphi \rangle
\end{align*}
for all vector fields $X$, $Y$.
\end{lem}

In particular, any Riemannian spin manifold $(M,g)$ admitting a Killing spinor with Killing number $\mu$ is Einstein with $\Ric = 4 \mu^2(n-1) g$. Another interesting aspect of Killing spinors is that they correspond to the limit case of the Friedrich \spinc inequality for Dirac eigenvalues.

\begin{thm}[Friedrich \spinc inequality]\label{Friedrichinequ}
On a compact Riemannian \spinc manifold $(M,g)$ the square of any eigenvalue $\lambda$ of $D^A$ is bounded from below by
\begin{equation}\label{FriedrichBound}
\lambda^2 \geq \frac{n}{4(n-1)} \inf_M \; (\Scal - c_n \vert \Omega \vert_g ),
\end{equation}
with $c_n = 2 \left[ \frac{n}{2} \right]^{\frac{1}{2}}$. Here the norm of a $2$-form $\omega$ is defined by
\begin{align*}
\vert \omega \vert^2_g \coloneqq \sum_{i < j} \vert \omega(e_i , e_j) \vert^2
\end{align*}
in any orthonormal frame. Furthermore, equality holds if and only if the corresponding eigenspinor $\varphi$ is a Killing spinor and $\Omega \cdot \varphi = i \frac{c_n}{2} \vert \Omega \vert_g \varphi$.
\end{thm}

This bound follows immediately from  \cite[Lemma 3.3]{Herzlich} and the Schr\"odinger-Lichnerowicz formula \eqref{SLformula}.

For later use we modify the connection $\spinccon$ in the following way:

\begin{defi}[Friedrich connection]
The Friedrich connection associated to $\mu \in \mathbb{C}$ is defined as
\begin{align*}
\fried_X \varphi \coloneqq \spinccon_X \varphi - \mu X \cdot \varphi,
\end{align*}
for any spinor $\varphi$ and vector field $X$.
\end{defi}

The Friedrich connection is metric if and only if $\mu$ is real. In particular Killing spinors with Killing number $\mu$ are parallel with respect to $\fried$. A straight-forward calculation leads to a version of \eqref{SLformula} for the Friedrich connection.

\begin{lem}\label{modSLformula}
For any real number $\mu$ and any spinor $\varphi$
\begin{align*}
(D^A + \mu)^2 \varphi = \friedspinclaplace \varphi + \frac{1}{4} \Scal \, \varphi + \frac{i}{2} \Omega \cdot \varphi -\mu^2 (n-1) \varphi .
\end{align*}
\end{lem}

\section{Identifying metric $\text{spin}^c$ structures}\label{IdentifyStructures}

Let $M$ be a \spinc manifold with a fixed topological \spinc structure. It induces for any Riemannian metric $g$ on $M$ a metric \spinc structure. Assume now that we have two different metrics $g$ and $h$ on $M$. Then the topological \spinc structure on $M$ descends to two different metric \spinc structures. Following \cite{IdentifySpin}, we construct an isomorphism between them and study its properties. 

For two Riemannian metrics $g$ and $h$ on $M$ there exists a unique positive definite endomorphism field $H_g$ such that $g(H_g X,Y) = h(X,Y)$ for all vector fields $X$ and $Y$. Its unique positive definite square root $b_g^h \coloneqq \sqrt{H_g}$ satisfies $g(b_g^h X, b_g^h Y) = h(X,Y)$ for all vector fields $X$, $Y$.

Since the topological \spinc structure on $M$ is fixed, both induced metric \spinc structures are build with the same principal $\sphere^1$-bundle $P$. Therefore, the map
\begin{align*}
(b_g^h)^n \times \Id: P_{\So} (M,h) \times P & \longrightarrow P_{\So}(M,g) \times P\\
(e_1, \ldots , e_n, s) & \longmapsto (b_g^h e_1 , \ldots , b_g^h e_n , s)
\end{align*}
lifts to a $\Spinc (n)$-equivariant isomorphism
\begin{align*}
\widetilde{(b_g^h)^n \times \Id} : P_{\Spinc} (M,h) & \rightarrow P_{\Spinc}(M,g).
\end{align*}
This induces the following isomorphism of \spinc bundles.
\begin{align*}
\beta_g^h: \spincibundle{h} = P_{\Spinc} (M,h) \times_{\delta} \Sigma_n & \rightarrow \spincibundle{g} = P_{\Spinc} (M,g) \times_{\delta} \Sigma_n \\
\varphi = [s, \psi] & \mapsto \beta_g^h \varphi \coloneqq [\widetilde{(b_g^h)^n \times \Id} (s) , \psi].
\end{align*}

Now we want to compare the action of the Levi-Civita connections $\conng$ and $\connh$. Define the connection $\comcon{h}{g}{Y} X \coloneqq b_h^g( \conng_Y (b_g^h X) )$ for any vector field $X$. Note that $(b_g^h)^{-1} = b_h^g$. Straight-forward calculations lead to:

\begin{lem}\label{propcomcon}
The torsion $\prescript{g}{}{T}^h$ of $\comcon{h}{g}{\,}$ satisfies
\begin{align*}
\prescript{g}{}{T}^h (X,Y) & = b_h^g ( (\conng_X b_g^h) Y  - (\conng_Y b_g^h) X), \, \text{and} \\
2h(\comcon{h}{g}{X} Y - \connh_X Y , Z) &= h (\prescript{g}{}{T}^h(X,Y), Z)  - h(\prescript{g}{}{T}^h(Y,Z) , X)\\ 
& \ \  - h(\prescript{g}{}{T}^h (Z,X) , Y).
\end{align*}
\end{lem}

Assuming $M$ to be \spinc the spinorial connection $\spincicon{g}$ on $\spincibundle{g}$ is built with the connection 1-form $A_g$ on $P$ and the spinorial connection on $\spincibundle{h}$ with the connection 1-form $A_h$. Define the connection $ \comcon{A_h}{A_g}{X} \varphi \coloneqq \beta_h^g(\spincicon{g}_X (\beta_g^h \varphi))$.

\begin{prop}\label{ConnCompare}
Given two metrics $g$ and $h$ on a fixed $n$-dimensional \spinc manifold $M$, there is a positive constant $C(n)$ such that
\begin{align*}
\vert \comcon{A_h}{A_g}{X} \varphi - \spincicon{h} \varphi \vert_h \leq C(n) \Vert X \Vert_h \; \Vert \varphi \Vert_h \left( \Vert \beta_h^g \Vert_h \; \Vert \conng \beta_g^h \Vert_h + \Vert A_h - A_g \Vert_h \right)
\end{align*}
for all vector fields $X$ and spinors $\varphi$.
\end{prop}

\begin{proof}
Let $\onframe$ be a local orthonormal frame with respect to $h$. Then any vector field can be written as $Y = \sum_{a=1}^n y^a e_a$. For the comparison of the associated spinorial connections we need to choose a local section $s$ of $P$. Using the local structure of spinorial connections, see for instance \cite[p.\ 60]{Friedrich}, we find
\begin{align*}
\vert \comcon{A_h}{A_g}{X} \varphi - \spincicon{h}_X \varphi \vert_h & \leq \frac{1}{4} \summeab \vert h ( \comcon{h}{g}{X} e_a - \connh_X e_a ) e_a \cdot e_b \cdot \varphi \vert_h \\
& \ \ \ + \frac{1}{2} \vert s^{\ast}(A_g - A_h)(X) \varphi \vert_h \\
& \leq C(n) \vert X \vert_h \vert \varphi \vert_h ( \Vert \beta_h^g \Vert_h \Vert \conng \beta_g^h \Vert_h + \Vert A_g - A_h \Vert_h ), 
\end{align*}
for all vector fields $X$ and spinors $\varphi$. The last inequality follows from Lemma \ref{propcomcon}.
\end{proof}

\begin{rem}
Similarly, one proves that for any two given Riemannian metrics $g$ and $h$ on an  $n$-dimensional manifold there is a positive constant $C^{\prime}(n)$ such that
\begin{align*}
\vert \comcon{h}{g}{X} Y - \connh_X Y\vert_h \leq C^{\prime}(n) \Vert b_h^g \Vert_h \; \Vert \conng b_g^h \Vert_h \; \Vert X \Vert_h \; \Vert Y \Vert_h
\end{align*}
for all vector fields $X$ and $Y$.
\end{rem}

\section{Convergence}\label{Convergence}

The goal of this section is to understand the convergence of almost Killing spinor sequences on \spinc manifolds. For this we need to establish a notion of convergence for \spinc manifolds. Since \spinc structures are built over the product of the frame bundle and a given principal $\sphere^1$-bundle we first need to develop a notion of convergence for sequences of manifolds with principal $\sphere^1$-bundles. This will be the content of Section \ref{CircleBundleConv}. There we show that a convenient uniform bound on the curvature of the principal $\sphere^1$-bundle leads to the existence of a subsequence such that the corresponding connection 1-forms converge after suitable choices of gauge transformations. Combining this with the known compactness results of manifolds we derive a useful notion of convergence.

In Section \ref{AlmostKillingSpinorSolutions}, we  define almost Killing spinor sequences on \spinc manifolds. By the convergence results of Section \ref{CircleBundleConv} it is enough to consider a manifold with a fixed topological \spinc structure. Therefore, we only need a few modifications due to the chosen connection 1-form $A$ on the principal $\sphere^1$-bundle.

\subsection{Convergence of principal $\sphere^1$-bundles}\label{CircleBundleConv}

The goal of this section is to establish a general notion of convergence for principal $\sphere^1$-bundles with connection. Note that we do not assume the manifold to be \spinc. First we show that for a suitable bound on the curvature of the principal $\sphere^1$-bundle there are only finitely many possibilities of isomorphism classes of principal $\sphere^1$-bundles satisfying it. Thus, we can focus on a sequence of connection 1-forms on a fixed principal $\sphere^1$-bundle where we obtain a converging subsequence by applying suitable gauge transformations.

In the beginning, we recall the basic classification results for principal $\sphere^1$-bundles.  For more details see e.g.\ \cite[Chapter 2]{Blair} and \cite[Chapter VI]{Brylinski}. These results are the main ingredient to prove the desired convergence results. 

Although this section does not require manifolds to be \spinc, we will stick to the notation used so far. Recall the following terminology: Two principal $\sphere^1$-bundles $P$ and $P^{\prime}$ together with connections $A$ resp.\ $A^{\prime}$ are \textit{isomorphic with connections} if there is a principal bundle isomorphism $\Phi:P \rightarrow P^{\prime}$ such that $\Phi^{\ast} A^{\prime} = A$. 

Isomorphism classes of principal $\sphere^1$-bundles as well as gauge equivalence classes are classified by the \v{C}ech-cohomology of the underlying base manifold $M$. Especially the classification of isomorphism classes is a well-known result which we restate here.

\begin{thm}\label{IsoClassify}
Let $M$ be a compact manifold. Then there is a bijection between the \v{C}ech-cohomology group $\check{\mathrm{H}}^2(M, \bbZ)$ and the isomorphism classes of principal $\sphere^1$-bundles over $M$. 
\end{thm}

Let $P$ be a principal $\sphere^1$-bundle over a compact manifold $M$. Then $P$ defines a unique class in $\check{\mathrm{H}}^2(M, \bbZ)$. This class is called the \textit{first Chern class} of $P$.

The curvature of a connection 1-form $A$ on $P$ is given by a closed 2-form $\Omega$ on $M$, namely
\begin{align*}
\mathrm{d} A = F_A = i \Omega.
\end{align*}
The de Rham class $[ - \frac{1}{2 \pi}\Omega ] \in \mathrm{H}^2(M, \bbR)$ is the image of the first Chern class of $P$ under the \v{C}ech-de Rham isomorphism. A short calculation shows that $[ - \frac{1}{2 \pi} \Omega ]$ is independent of the choice of the connection 1-form $A$ on $P$. Thus, it depends only on the isomorphism class of the principal $\sphere^1$-bundle.

As mentioned in the introduction, we want to show that there is a suitable bound on the curvature such that there are only finitely many principal $\sphere^1$-bundles up to isomorphism satisfying it. For this recall from Hodge theory that on a compact Riemannian manifold $(M,g)$ each de Rham class $[\omega]$ admits a unique harmonic representative $\widetilde{\omega}$. Moreover, $\widetilde{\omega}$ minimizes the $L^2$-norm in the class $[\omega ]$. In addition, the projection from closed to harmonic forms is continuous in $L^2.$ Thus, it is a natural choice to assume an $L^2$-bound on the curvature for our purpose .

\begin{lem}\label{boundedCurv}
Let $(M,g)$ be a compact Riemannian manifold and $K$ a fixed non-negative number. Then there are only finitely many isomorphism classes of principal $\sphere^1$-bundles $P$ with connection over $M$ whose curvature satisfies $\Vert \Omega \Vert_{L^2} \leq K$.
\end{lem}

\begin{proof}
By Theorem \ref{IsoClassify} the isomorphism classes of principal $\sphere^1$-bundles over $M$ are classified by $\check{\mathrm{H}}^2(M, \bbZ)$. By the universal coefficient theorem, we have $\check{\mathrm{H}}^2(M, \bbZ) \cong \bbZ^{b_2(M)} \oplus T_1$, where $T_1$ is the torsion of $\check{\mathrm{H}}_1(M, \bbZ)$ which is finite, and $b_2(M)$ the second Betti number of $M$. Furthermore, the kernel of the homeomorphism $h:\check{\mathrm{H}}^2(M, \bbZ) \rightarrow \mathrm{H}^2(M, \bbR)$ is given by $T_1$. The cohomology class $[ - \frac{1}{2 \pi}\Omega] \in \mathrm{H}^2(M, \bbR)$ is an integral class, i.e.\ it lies in the image of $h$.

Thus, the set of isomorphism classes of principal bundles whose curvature satisfies $\Vert \Omega \Vert_{L^2} \leq K$ is given by $h^{-1}(C)$, where 
\begin{align*}
C \coloneqq \Bigg\lbrace \left[ \omega \right] \in \mathrm{H}^2(M, \bbR) : \, \Vert \widetilde{\omega} \Vert_{L^2} \leq \frac{1}{2 \pi} K \Bigg\rbrace.
\end{align*}
Here $\widetilde{\omega}$ denotes the unique harmonic representative of $[ \omega]$. 

Since $\mathrm{H}^2(M, \bbR) \cong \mathcal{H}^2(M) \cong \bbR^{b_2(M)}$, with $\mathcal{H}^2(M)$ denoting the space of harmonic 2-forms, is a finite dimensional vector space and the projection from closed to harmonic forms is continuous in $L^2$, it follows that $C$ is compact. In particular, $\Image(h) \cap C$ is compact, hence finite. Since the kernel of $h$ is also finite the claim follows.
\end{proof}

We recall now the characterization of the gauge equivalence classes of connections on a fixed principal $\sphere^1$-bundle $P$ over $M$ which can be found in the standard literature.

\begin{thm}\label{equiv}
For a fixed principal $\sphere^1$-bundle $P$ over a compact Riemannian manifold $M$ two principal connections are gauge equivalent if and only if their difference is represented by a closed integral 1-form. In particular, the space of gauge equivalence classes of connections with fixed curvature $\Omega$ is given by the Jacobi torus $\check{\mathrm{H}}^1(M, \bbR) / \check{\mathrm{H}}^1(M, \bbZ)$.
\end{thm}

Using this theorem we are able to prove the following convergence result. However, note that we will in general not obtain $C^{\infty}$-convergence. Therefore, we establish here the following notion: A connection 1-form $A$ is called \Cka{k} if its associated Christoffel symbols are \Cka{k}. Further on, we only consider $\alpha \in (0,1)$.

\begin{thm}\label{fixedConvergence}
Let \sequence{P_j , A} be a sequence of principal $\sphere^1$-bundles with connection over a fixed compact Riemannian manifold $(M,g)$. For each $j$ let $\Omega_j$ be the 2-form representing the curvature of $A_j$. If there is a non-negative $K$ such that $\Vert \Omega_j \Vert_{C^{k, \alpha}} \leq K$ for all $j$, then for any $\beta <\alpha$ there are a principal $\sphere^1$-bundle $P$ with a $C^{k+1, \beta}$-connection $A$ and a subsequence, again denoted by \sequence{P_j , A} together with principal bundle isomorphisms $\Phi_j : P \rightarrow P_j$ such that $\Phi_j^{\ast} A_j$ converges to $A$ in the $C^{k+1 , \beta}$-norm
\end{thm}

\begin{proof}

Since the $C^{k,\alpha}$-norm of the curvatures of the principal $\sphere^1$-bundle~$P_j$ is uniformly bounded in $j$, the $L^2$-norm of the curvatures is also uniformly bounded. Applying Lemma \ref{boundedCurv} we conclude that this sequence of principal $\sphere^1$-bundles only contains finitely many isomorphism classes. Hence, we find a subsequence \sequence{P_j, A} such that there is for each $j$ an isomorphism $\Psi_j: P_j \rightarrow P$ for some fixed $P$. Therefore, it is

Using that the connections on $P$ form an affine space over $\Gamma(\Lambda^1 T^{\ast} M)$, we will fix $A_1$ as a reference connection. The difference $\Psi_j^{\ast}A_j~-~\Psi_1^{\ast}A_1$ is given by a unique $i \eta_j$ with $\eta_j \in \Gamma(\Lambda^1 T^{\ast} M)$. We will apply the Hodge decomposition various times and show for each part separately how we obtain a converging subsequence.

Since $P$ is fixed $\left[ - \frac{1}{2\pi} \Psi_j^{\ast}\Omega_j \right] = \left[  - \frac{1}{2\pi} \Psi_k^{\ast}\Omega_k \right]$ for all $j$, $k$. Hence, for each $j$ there is  a 1-form $\zeta_j$ such that 
\begin{align*}
\Psi_j^{\ast}\Omega _j = \Psi_1^{\ast} \Omega_1 + \mathrm{d}\zeta_j. 
\end{align*}
By our assumptions on the curvatures there is a positive constant $\widetilde{K}$ such that $\Vert \mathrm{d}\zeta_j \Vert_{C^{k, \alpha}} \leq \widetilde{K}$ uniformly $j$. 

By Hodge decomposition we can choose $\zeta_j = \delta \xi_j$ for some closed 2-form $\xi_j$ which is orthogonal to $\ker (\Delta)$ in $L^2$. Thus, $\mathrm{d} \zeta_j = \Delta \xi_j$. By Schauder's estimate we find a positive constant $C$ such that for all $j$
\begin{align*}
\Vert \xi_j \Vert_{C^{k+2, \alpha}} \leq C \Vert \Delta \xi_j \Vert_{C^{k, \alpha}} \leq C \widetilde{K}.
\end{align*}
For any $\beta <\alpha$ there is a subsequence \sequence{\xi} converging in $C^{k+2, \beta }$. Thus, \sequence{\delta \xi} converges in $C^{k+1, \beta}$. In general, the limit is not smooth.

For each $j$ the connections $\Psi_j^{\ast} A_j$ and $\Psi_1^{\ast}A_1 + i \delta \xi_j$ have the same curvature. Thus, for each $j$, there is a unique closed 1-form $\eta_j$ such that
\begin{align*}
\Psi_j^{\ast} A_j = \Psi_1^{\ast} A_1 + i \delta \xi_j + i \eta_j.
\end{align*}
Again we apply the Hodge decomposition and obtain for each $j$ a smooth function $f_j$ and a harmonic 1-form $\nu_j$ such that
\begin{align*}
\eta_j = \mathrm{d}f_j + \nu_j.
\end{align*}
If $\mathrm{d}f_j \neq 0$ we apply the gauge transformation $G_j = e^{- i f_j}$ and obtain
\begin{align*}
G_j^{\ast} \Psi_j^{\ast} A_j = \Psi_1^{\ast} A_1 + i \delta \xi_j + i \nu_j.
\end{align*}
Now, we need to find a subsequence and suitable gauge transformations such that the remaining harmonic parts converge. To obtain these we take a closer look at  the classification of connections on a principal $\sphere^1$-bundle. By Theorem \ref{equiv}, the gauge equivalence classes of connections for a fixed curvature form are classified by the Jacobi torus $\check{\mathrm{H}}^1(M, \bbR) / \check{\mathrm{H}}^1(M, \bbZ)$. By Hodge theory, there is exactly one harmonic representative in each de Rham class. Since $\check{\mathrm{H}}^1(M, \bbZ)$ has no torsion elements it is embedded in $\mathrm{H}^1(M, \bbR)$ via the \v{C}ech-de Rham isomorphism. Hence, we obtain the quotient of harmonic forms divided by harmonic integral forms which is isomorphic to the torus $\bbT^{b_1(M)}$. As the projection from closed to harmonic forms is continuous in $L^2$ the Jacobi torus is compact in the $L^2$-topology.

The sequence \sequence{\nu} induces a sequence in the Jacobi torus. Since it is a compact quotient in the $L^2$-topology there is a subsequence of harmonic representatives \sequence{\widetilde{\nu}} converging in $L^2$ to a smooth harmonic 1-form $\widetilde{\nu}$. Note that each $\nu_j$ is equivalent to $\widetilde{\nu}_j$. By standard elliptic estimates it follows that \sequence{\widetilde{\nu}}  converges in $C^l$ for any $l > 0$. 

Taking the corresponding gauge transformations $H_j$, we obtain the sequence
\begin{align*}
\left( H_j^{\ast} G_j^{\ast} \Psi_j^{\ast} A_j = \Psi_1^{\ast} A_1 + i \delta_1 \xi_j + i \widetilde{\nu_j} \right)_{j \in \bbN},
\end{align*}
which converges in $C^{k+1, \beta}$. Setting $\Phi_j := \Psi_j \circ G_j \circ H_j$ finishes the proof.
\end{proof}

\begin{rem}
Similarly a uniform upper bound on the $W^{k,2}$-norm of the curvature leads to a $W^{l+1,2}$-converging subsequence of the underlying connection 1-forms for any $l < k$. 
\end{rem}

This theorem shows that the space of principal $\sphere^1$-bundles over a fixed compact Riemannian manifold $(M,g)$ with a uniform bound on the $C^{k, \alpha}$-norm of the curvature is ``precompact in the $C^{k+1, \beta}$-topology" for any $\beta <\alpha$. Now we also want to vary the base manifold $(M,g)$. For this we use the following compactness theorem by Anderson, \cite[Theorem 1.1]{Anderson}.

\begin{thm}\label{CompactSpace}
For given positive numbers $\Lambda$ , $i_0$, and $d$ the class $\mfdspace$ of Riemannian $n$-manifolds with
\begin{align*}
\vert \Ric \vert \leq \Lambda , \quad \inj \geq i_0 , \quad \diam \leq d,
\end{align*}
is compact in the \Cka{1}-topology for any $\alpha \in (0,1)$. Furthermore, the subspace consisting of Einstein manifolds is compact in the $C^{\infty}$-topology.
\end{thm}

Combining this class of manifolds with the assumptions in Theorem~\ref{fixedConvergence} we define the following class of manifolds with principal $\sphere^1$-bundles.

\begin{defi}\label{precompRic}
Let $\precompRic$ be the space of principal $\sphere^1$-bundles $P \stackrel{\pi}{\longrightarrow} M$ with principal connection $A$ such that $(M,g)$ lies in $\mfdspace$ and $\Vert \Omega \Vert_{C^{0, 1}(g)} \leq K$ where $\Omega \in \Gamma(\Lambda^2 T^{\ast}M)$ is the curvature of $A$, i.e. $i \Omega= F_A$.
\end{defi}

\begin{thm}\label{CompPrincSet}
Any sequence \sequence{M_j, g_j, P_j, A} in $\precompRic$ admits a subsequence, again denoted by \sequence{M_j, g_j, P_j, A}, such that for any $\alpha \in (0,1)$ there is a principal $\sphere^1$-bundle $P$ over a closed Riemannian manifold $M$ with a \Cka{1}-metric $g$ and a \Cka{1}-connection $A$ such that for each $j$ there is a principal bundle isomorphism
\begin{align*}
\begin{xy}
\xymatrix{
P  \ar[r]^{\Phi_j} \ar[d] & P_j \ar[d]\\
M  \ar[r]^{\phi_j} & M_j
}
\end{xy}
\end{align*}
with $\Phi_j^{\ast} A_j$ and $\phi_j^{\ast} g_j$ converging to $A$ resp.\ $g$ in \Cka{1}.
\end{thm}

\begin{proof}
Let \sequence{M_j, g_j, P_j, A} be a sequence in $\precompRic$. Since any manifold in this sequence lies in $\mfdspace$ there exists a subsequence again denoted by \sequence{M_j, g_j, P_j, A} and a Riemannian manifold $(M,g)$ such that for each $j$ there exists a diffeomorphism $\phi_j: M \rightarrow M_j$ with $\phi_j^{\ast} g_j$ converging to $g$ in $C^{1, \alpha}$ by Theorem \ref{CompactSpace}.

By pulling back each element in \sequence{M_j, g_j, P_j, A} with the diffeomorphism $\phi_j$ we obtain a sequence of metrics and principal $\sphere^1$-bundles with connections over a fixed compact manifold $M$ which we call \sequence{M, g_j, P_j, A} for simplicity.

We fix the initial metric $g_1$ as our background metric. Applying Theorem \ref{fixedConvergence} to the sequence \sequence{P_j, A} viewed over $(M, g_1)$ we obtain a subsequence together with principal bundle isomorphism $\Psi_j: P \rightarrow P_j$ such that $\Psi_j^{\ast} A_j$ converges in $C^{1, \alpha}(g_1)$.

Since \sequence{g} converges to $g$ in $C^{1, \alpha}$ the claim follows.
\end{proof}

\subsection{Almost Killing spinor sequences}\label{AlmostKillingSpinorSolutions}

An almost Killing spinor sequence describes a converging sequence of metrics on a fixed spin manifold together with a sequence of spinors converging to a nontrivial Killing spinor. This concept was first introduced for spin manifolds in \cite{Vargas}. We extend it to \spinc manifolds. By the results of Section \ref{CircleBundleConv}, it is enough to look at manifolds with fixed topological \spinc structure.

\begin{defi}[Almost Killing spinor sequence]
Let \sequence{M, g_j, P, A} be a sequence on a fixed Riemannian \spinc manifold whose topological \spinc structure is built with $P$ such that \sequence{g} converges in $C^1$ and \sequence{A} in $C^0$. A sequence \sequence{\varphi} of spinors $\varphi \in W^{1,2}(\spincibundle{j})$ is an \textit{almost Killing spinor sequence} if there exists a real $\mu$ and a vanishing sequence \sequence{\eps} such that
\begin{align*}
\Norm{\friedi \varphi_j}{\Lspinci} \leq O(\eps_j) \, \Norm{\varphi}{\Lspinci}
\end{align*}
for all $j$. An almost Killing spinor sequence called \textit{$L^2$-normalized} if $\Norm{\varphi_j}{\Lspinci}=1$ for all $j$.
\end{defi}

In the end, we want to show that an $L^2$-normalized almost Killing spinor sequence converges strongly in $W^{1,2}$ to a nontrivial Killing spinor. Since we only need the embedding $\sob \hookrightarrow \Lspinc$ to be compact these convergence results apply to any Riemannian \spinc manifold on which the Sobolev embedding theorems hold.

Since almost Killing spinor sequences are defined via the Friedrich connection we define the norm 
\begin{align*}
\Norm{\varphi}{\friedsob}^2 \coloneqq \Norm{\varphi}{\Lspinc}^2 + \Norm{\fried \varphi}{\Lspinc}^2
\end{align*}
on $\sob$. This norm is equivalent to the Sobolev norm
\begin{align*}
\Norm{\varphi}{\sob}^2 = \Norm{\varphi}{\Lspinc}^2 + \Norm{\spinccon \varphi}{\Lspinc}^2.
\end{align*}

The next two results generalize Lemma 4.5.1 and Theorem 4.5.2 in \cite{Vargas}. We follow the original proofs and generalize them to the \spinc case. 

\begin{lem}\label{AkssSobolev}
Let \sequence{\varphi} be an almost Killing spinor sequence. Then \sequence{\beta_g^{g_j} \varphi} is bounded in $\sob$, where $\beta_g^{g_j}$ is the endomorphism field relating $g_j$ with $g$. Moreover, $\lim_{j \rightarrow \infty} \Norm{\beta_g^{g_j} \varphi_j}{\friedsob}= 1$.
\end{lem}

\begin{proof}
Since the norms $\Norm{\cdot }{\friedsob}$ and $\Norm{\varphi}{\sob}$ are equivalent, we only need to prove that \sequence{\beta_g^{g_j} \varphi} is bounded in $\Norm{\cdot}{\friedsob}$. 
\begin{align*}
\Norm{\beta_g^{g_j} \varphi_j}{\Lspinc}^2 & = \int_M \langle \beta_g^{g_j} \varphi_j , \beta_g^{g_j} \varphi_j \rangle_g \dvol{g} \\
& = \int_M \langle \varphi_j , \varphi_j \rangle_{g_j} \vert \det (\beta_g^{g_j}) \vert^{-1} \dvol{g_j} \\
& \leq (1 + O(\eps_j) ) \Norm{\varphi_j}{\Lspinci}^2,
\end{align*}
as $\beta_g^{g_j}$ converges to the identity in $C^1$. Similarly we obtain the estimate $\Norm{\beta_g^{g_j} \varphi_j}{\Lspinc}^2 \geq (1 - O(\eps_j) ) \Norm{\varphi_j}{\Lspinci}^2$.

Now we need an upper bound on the second part of $\Norm{\cdot}{\friedsob}$. 
\begin{align*}
\Norm{\fried (\beta_g^{g_j} \varphi_j)}{\Lspinc}^2 & = \int_M \langle \fried (\beta_g^{g_j} \varphi_j) , \fried (\beta_g^{g_j} \varphi_j) \rangle_{g} \dvol{g} \\
& = \int_M \langle \beta_{g_j}^g  \fried (\beta_g^{g_j} \varphi_j) , \beta_{g_j}^g  \fried (\beta_g^{g_j}\varphi_j) \rangle_{g_j} \dvol{g_j} \\
& \leq (1 + O(\eps_j) )\Norm{\friedcomcon \varphi_j}{\Lspinci},
\end{align*}
where we define $\friedcomcon_X \psi \coloneqq \beta_{g_j}^g \fried_X (\beta_g^{g_j} \psi)$ for all vector fields $X$ and spinors $\psi$. We estimate further
\begin{align*}
&\Norm{\friedcomcon \varphi_j}{\Lspinci}\\ 
& \ \ \ \leq \Norm{\friedcomcon \varphi_j - \friedi \varphi_j}{\Lspinci} + \Norm{\friedi \varphi_j}{\Lspinci} \\
& \ \ \ \leq \Norm{\comcon{A_j}{A}{} \varphi_j - \spincicon{j} \varphi_j}{\Lspinci} + \Norm{\beta_{g_j}^g  - \Id}{g_j} \Norm{\varphi_j}{\Lspinci} \\ 
&  \ \ \ \ \ \ + O(\eps_j) \Norm{\varphi_j}{\Lspinci} \\
&\ \ \ \leq C \Big( \Vert \beta_{g_j}^g  \Vert_{g_j} \Vert \conng \beta_g^{g_j} \Vert_{g_j} + \Vert A - A_{j} \Vert_{g_j}  + O(\eps_j) \Big ) \Norm{\varphi_j}{\Lspinci}\\
& \ \ \ \leq O(\eps_j) \Norm{\varphi_j}{\Lspinci},
\end{align*}
where we applied Proposition \ref{ConnCompare} and used the $C^1$-convergence of $\left( \beta_g^{g_j} \right)_{j \in \bbN}$.
\end{proof}

This lemma together with the Sobolev embedding theorem leads to the following convergence result in $W^{1,2}$ generalizing \cite[Theorem~4.5.2]{Vargas} which we just state, as the proof is the same.

\begin{prop}\label{Killconverge}
Any $L^2$-normalized almost Killing spinor sequence admits a subsequence converging strongly in $W^{1,2}$ to a nontrivial $L^2$-Killing spinor $\varphi$, i.e.\ $\Norm{\fried \varphi}{\Lspinc} = 0$.
\end{prop}

\section{Regularity}\label{Regularity}

In this section we study the relation between Killing spinors and the regularity of the Riemannian metric $g$ and the connection 1-form $A$. We mostly use elliptic regularity. 

We start by observing that the regularity of a Killing spinor depends on the regularity of the connection $\spinccon$ and therefore on the regularity of the metric $g$ and the connection 1-form $A$. A Killing spinor with Killing number $\mu$ lies in the kernel of the elliptic second order operator $\friedspinclaplace$ whose coefficients contain the first derivatives of the coefficients of $g$ and $A$. Therefore we apply elliptic regularity and obtain:

\begin{prop}\label{KillSpinRegu}
Let $(M,g)$ be a Riemannian \spinc manifold with $g$ and $A$ being \Cka{k}, $k \geq 1$. Then any $L^2$-Killing spinor $\varphi$ on $(M,g)$ is \Cka{k + 1}. 
\end{prop}

Recall that any Riemannian spin manifold admitting a nontrivial Killing spinor is Einstein. Therefore, its metric is real analytic in harmonic coordinates by \cite[Theorem 5.2]{DeTurckKazdan}. However, on Riemannian \spinc manifolds with a Killing spinor the Ricci curvature is given by the identity in Lemma \ref{KillRic}. We use the rest of this section to study the relation between the regularity of $A$, or more precisely the regularity of its curvature form $\Omega = - i F_A$, and the regularity of the metric $g$ on a \spinc manifold with a Killing spinor in harmonic coordinates.

Since we want to include \Cka{1}-metrics we need to generalize the required equations to a weak context. In particular, we want to obtain an analogue of Lemma \ref{KillRic} for \Cka{1}-metrics. As the Ricci curvature and the spinorial curvature involve second derivatives of the metric we need to redefine these objects in a weak context. Further on, $\langle . , .\rangle_U$ denotes the pairing between a distribution and a test function compactly supported on $U \subset M$.

Using the coordinate description of the Ricci curvature (see \cite[Lemma~4.1]{DeTurckKazdan}) and Stokes' theorem we define a weak version of the Ricci curvature as follows. In a local chart $U$ with coordinate vector fields $\coordframe$ the weak Ricci curvature acts on a test function $\eta \in C^1_c (U)$ by
\begin{align*}
\langle \Ric(X, Y) , \eta \rangle_U \coloneqq \summeab \langle \Ric_{ab}, X^a Y^a \eta \rangle_U.
\end{align*}
Here $X = \summea X^a \partial_a$ and $Y = \summeab Y^a \partial_a$ are local vector fields. The right hand side is defined by
\begin{align*}
\langle \Ric_{ab},\eta \rangle_U & \coloneqq \int_U \Big( \partial_s g_{ab} \, \partial_r (g^{rs} \sqrt{g} \, \eta) - \Gamma^r \, \partial_b(g_{ra} \sqrt{g} \, \eta) - \Gamma^r \, \partial_a(g_{rb} \sqrt{g} \, \eta)\\
& \qquad \quad \ + Q_{ab} \, \sqrt{g} \, \eta \Big) \, \mathrm{d}x ,
\end{align*}
where $\sqrt{g}$ denotes the square root of the determinant of the metric $g$ in these coordinates and $Q_{ab}$ denotes a quadratic form depending only on $g$ and its derivatives $\partial g$.

For the spinorial curvature, we again use Stokes' theorem to obtain a weak version. Since the inner product on $\spincbundle$ is Hermitian, we consider each entry separately. This leads to the following definition for the weak spinorial curvature for vector fields $X$, $Y$, and a spinor $\varphi \in C^1(\spincbundle)$ acting on a test spinor $\psi \in \hoelderspincu$,
\begin{align*}
\langle \R^A(X,Y) \varphi, \psi \rangle_U & \coloneqq \int_U \langle \spinccon_X \varphi , \spinccon_Y \hat{\psi} \rangle - \langle \spinccon_Y \varphi , \spinccon_X \hat{\psi} \rangle - \langle \spinccon_{[X, Y]} \varphi ,  \hat{\psi} \rangle \mathrm{d}x, \\
\langle \psi , \R^A(X,Y) \varphi \rangle_U &\coloneqq \int_U \langle \spinccon_Y \hat{\psi} , \spinccon_X \varphi \rangle - \langle \spinccon_X \hat{\psi} , \spinccon_Y \varphi \rangle - \langle  \hat{\psi} , \spinccon_{[X,Y]} \varphi \rangle \mathrm{d}x,
\end{align*}
in any local chart $U$ of $M$. Here we set $\hat{\psi} \coloneqq \sqrt{g} \psi$. 

Now, we show that the relation \eqref{RicSpin} in Section 2 between Ricci and spinorial curvature also holds in the weak context. This generalizes Lemma 3.4.1 in \cite{Vargas}.

\begin{lem}\label{WeaklyRicciSpin}
On a Riemannian \spinc manifold $M$ with a $C^1$-metric $g$ and $C^0$-curvature form $\Omega$
\begin{align*}
\langle \Ric(X) \cdot \varphi , \psi \rangle_U & \coloneqq \summei \langle \Ric(X, e_i), \langle e_i \cdot \varphi, \psi \rangle \ \rangle_U \\
&= 2 \summei \langle \R^A(X, e_i)\varphi, e_i \cdot \psi \rangle_U + i \int_U \langle (X \lrcorner \Omega) \cdot \varphi, \psi \rangle \dvol{g}
\end{align*}
for $X \in \mathfrak{X}(M)$, $\varphi \in C^1(\spincbundle)$ and $\psi \in \hoelderspincu$ in any orthonormal frame $\onframe$.
Similarly,
\begin{align*}
\langle \psi , \Ric(X) \cdot \varphi \rangle_U & \coloneqq \summei \langle \Ric(X, e_i), \langle \psi , e_i \cdot \varphi \rangle \ \rangle_U \\
& \ = 2 \summei \langle e_i \cdot \psi, \R^A(X, e_i)\varphi \rangle_U + i \int_U \langle (X \lrcorner \Omega) \cdot \psi, \varphi \rangle \dvol{g}.
\end{align*}
\end{lem}

\begin{proof}
We integrate equation \eqref{RicSpin} in Section 2 against a test spinor $\psi~\in~\hoelderspincu$ and obtain
\begin{equation}\label{term}
\begin{aligned}
\int_U \langle \Ric(X) \cdot \varphi, \psi \rangle \dvol{g} &= \int_U 2  \summei  \langle \R^A(X, e_i)\varphi, e_i \cdot \psi \rangle \dvol{g}\\
& \ \ \  + i\int_U \langle (X \lrcorner \Omega) \cdot \varphi, \psi \rangle \dvol{g}
\end{aligned}
\end{equation}
for any Riemannian \spinc manifold with a $C^2$-metric and $\Omega$ being $C^0$.

We rewrite the left hand side as
\begin{align*}
\summei \int_U \Ric(X, e_i) \langle e_i \cdot \varphi, \psi \rangle \dvol{g} = \langle \Ric(X) \cdot \varphi, \psi \rangle_U
\end{align*}
and the first term on the right hand side of \eqref{term} as
\begin{align*}
\int_U 2  \summei  \langle \R^A(X, e_i)\varphi, e_i \cdot \psi \rangle \dvol{g} = 2 \summei \langle \R^A(X, e_i) \varphi, e_i \cdot \psi \rangle_U,
\end{align*}
to obtain the desired identity. 

Each term is well defined for $C^1$-metrics and spinors but so far we only know that this identity holds for $C^2$-metrics. Let $C^1(\mathcal{M}et(M))$ denote the class of $C^1$-metrics. Now, we consider each term separately as a map:
\begin{align*}
\mathcal{W}_{\R^A}: C^1(\spincbundle) \times C_c^1(\spincbundle) \times C^1(\mathcal{M}et(M)) & \rightarrow \bbC,\\
(\varphi, \psi, g) & \mapsto \langle \R^A(X, e_i) \varphi, e_i \cdot \psi \rangle_U, \\
\ & \ \\
\mathcal{W}_{\Ric}: C^1(\spincbundle) \times C_c^1(\spincbundle) \times C^1(\mathcal{M}et(M)) & \rightarrow \bbC,\\
(\varphi, \psi, g) & \mapsto \langle \Ric(X) \cdot \varphi, \psi \rangle_U,\\
\ & \ \\
 \mathcal{W}_{\Omega}: C^1(\spincbundle) \times C_c^1(\spincbundle) \times C^1(\mathcal{M}et(M)) & \rightarrow \bbC,\\
(\varphi, \psi, g) & \mapsto \int_U \langle (X \lrcorner \Omega) \cdot \varphi, \psi \rangle \dvol{g}.
\end{align*}
By Appendix B in \cite{Vargas}, $\mathcal{W}_{\R^A}$ and $\mathcal{W}_{\Ric}$ are continuous in the $C^1$-topology. The proofs are given for spin manifolds but extend easily to the \spinc case. The continuity in the $C^1$-topology of $\mathcal{W}_{\Omega}$ in the first two variables is obvious. A short calculation using the identification of \spinc structures with $\Omega$ explained in Section \ref{IdentifyStructures} shows that $\mathcal{W}_{\Omega}$ is also $C^1$-continuous in the third variable. Since smooth sections are dense in the considered domains the claim follows.
\end{proof}

Combining these two identities and inserting a Killing spinor generalizes Lemma \ref{KillRic} to a weak context. The statement is analogous to \cite[Theorem 3.4.2]{Vargas} but we choose another way of proving as there is a problem with the double usage of the index $i$ in \cite[p.~43]{Vargas}.

\begin{lem}\label{KillweakRic}
Let $(M,g)$ be a Riemannian \spinc manifold with $g$ being $C^1$, $\Omega$ being $C^0$ and a nontrivial Killing spinor $\varphi$ with Killing number $\mu$. Then
\begin{align*}
\langle \Ric_{ab} , \eta \rangle_U &= 4 \mu^2 (n-1) \int_U g_{ab}\eta \, \dvol{g} \\
& \ \ \ + \frac{i}{2}\int_U \frac{\eta}{\vert \varphi \vert^2} \langle (( \partial_a \lrcorner \Omega) \cdot \partial_b - \partial_b \cdot (\partial_a \lrcorner \Omega)) \cdot \varphi , \varphi \rangle \dvol{g}
\end{align*}
for any $\eta \in C^{\infty}_c (U)$ and $1 \leq a, b \leq n$ in any local chart $U$ with coordinate vector fields $\coordframe$. 
\end{lem}

\begin{proof}
On Riemannian \spinc manifolds with a $C^2$-metric 
\begin{align*}
\Ric(\partial_a , \partial_b) = \frac{1}{2 \vert \varphi \vert^2} \left( \langle \Ric(\partial_a) \cdot \varphi,  \partial_b \cdot \varphi \rangle + \langle \partial_b \cdot \varphi , \Ric(\partial_a) \cdot \varphi \rangle \right).
\end{align*}
Using the same method as in the proof of Lemma \ref{WeaklyRicciSpin}, we conclude that for $C^1$-metrics we have 
\begin{equation}\label{KillRicEq}
\begin{aligned}
\vert \varphi \vert^2 \langle \Ric_{ab} , \eta \rangle_U &= \frac{1}{2} \left( \langle \Ric(\partial_a) \cdot \varphi , \eta \, \partial_b \cdot \varphi \rangle_U \right. \\
& \ \ \ \ \ \ \left. + \langle \eta \, \partial_b \cdot \varphi , \Ric(\partial_a) \cdot \varphi \rangle_U \right) 
\end{aligned}
\end{equation}
for any test function $\eta \in C^{\infty}_c(U)$, spinor $\varphi$, and  $1 \leq a,b \leq n$.

By Lemma \ref{WeaklyRicciSpin}, we first study how the weak spinorial curvature acts on a Killing spinor $\varphi$ with Killling number $\mu$. A short calculation shows that
\begin{align*}
\langle R^A(X,Y)\varphi, \psi \rangle_U = \int_U \mu^2 \langle (Y \cdot X - X \cdot Y) \cdot \varphi, \psi \rangle_U \, \dvol{g},
\end{align*}
for any vector fields $X$ and $Y$ and test spinor $\psi \in \hoelderspincu$.

Inserting this into Lemma \ref{WeaklyRicciSpin} we obtain
\begin{align*}
&2 \summei \langle \R^A (\partial_a, e_i )\varphi, e_i \cdot \eta \, \partial_b \cdot \varphi \rangle_U \\
&  \ \ \ \ \ =2 \mu^2 \summei \int_U \eta \, \langle (e_i \cdot \partial_a \cdot e_i - e_i \cdot e_i \cdot \partial_a) \cdot \varphi , \partial_b \cdot \varphi \rangle  \, \dvol{g} \\
& \ \ \ \ \ = 4 \mu^2 \int_U \eta  \,\big\langle \Big( n\partial_a - \summei g(\partial_a , e_i ) \Big) \cdot \varphi , \varphi \big\rangle \, \dvol{g}\\
& \ \ \ \ \ = 4\mu^2 (n-1) \int_U \eta \,\langle \partial_a \cdot \varphi , \partial_b \cdot \varphi \rangle \, \dvol{g}.
\end{align*}

Thus,
\begin{align*}
\big\langle \Ric(\partial_a) \cdot \varphi , \eta \, \partial_b \cdot \varphi \big\rangle_U &= 4 \mu^2 (n-1) \int_U \eta \, \langle \partial_a \cdot \varphi , \partial_b \cdot \varphi \rangle \, \dvol{g}\\ & \ \ \ \ - i \int_U \eta \, \langle \partial_b \cdot (\partial_a \lrcorner \Omega) \cdot \varphi, \varphi \rangle \, \dvol{g}.
\end{align*}
A similar calculation leads to
\begin{align*}
\langle \eta \, \partial_b \cdot \varphi , \Ric(\partial_a) \cdot \varphi \rangle_U &= 4 \mu^2 (n-1) \int_U \eta\,  \langle \partial_b \cdot \varphi , \partial_a \cdot \varphi \rangle \, \dvol{g}\\
& \ \ \ \ + i \int_U \eta \, \langle (\partial_a \lrcorner \Omega) \cdot \partial_b \cdot \varphi, \varphi \rangle \eta \, \dvol{g}.
\end{align*}
Since $\varphi$ is a nontrivial Killing spinor it never vanishes. Thus, inserting these two identites into \eqref{KillRicEq} finishes the proof.
\end{proof}

Now, we switch to harmonic coordinates in which $g \mapsto \Ric(g)$ is an elliptic operator of second order. Applying elliptic regularity leads to:

\begin{thm}\label{Killregularity}
Let $(M,g)$ be a Riemannian \spinc manifold with a \Cka{1}-metric $g$, a \Cka{l}-curvature form $\Omega$ on $P$, $l \geq 0$, and a nontrivial Killing spinor $\varphi$. Then $g$ is \Cka{l+2} in harmonic coordinates.
\end{thm}

\begin{proof}
Let $(x_1 , \ldots , x_n)$ be harmonic coordinates in $U \subset M$. Then the components of the Ricci tensor are given by
\begin{align*}
\Ric_{ab} = - \frac{1}{2} g^{rs} \partial_r \partial_s g_{ab} + Q_{ab}(g, \partial g) ,
\end{align*}

By Lemma \ref{KillweakRic} $g_{ab}$ is a weak solution of the elliptic differential equation.
\begin{align*}
\int_U \Big( \partial_s (g_{ab}) \partial_r(g^{rs }\sqrt{g} \, \eta) - f_{ab} \sqrt{g} \, \eta \Big) dx = 0,
\end{align*}
where we set 
\begin{align*}
f_{ab} \coloneqq & -8 \mu^2(n-1)g_{ab} - \frac{i}{\vert \varphi \vert^2} \big\langle \big( ( \partial_a \lrcorner \Omega) \cdot \partial_b - \partial_b \cdot (\partial_a \lrcorner \Omega) \big) \cdot \varphi , \varphi \big\rangle\\
& + 2 Q(g, \partial g).
\end{align*} 
Note that $\varphi$ is \Cka{2} by Proposition \ref{KillSpinRegu}. Therefore $f_{ab}$ is \Cka{0} by our assumptions. By elliptic regularity (cf.\ \cite[Theorem 6.4.3]{Morrey}), $g_{ab}$ is in fact \Cka{2} in harmonic coordinates for all $1 \leq a,b \leq n$. In particular, the Ricci tensor is now well-defined by
\begin{align*}
\Ric_{ab} = 4 \mu^2 (n-1) g_{ab} + \frac{i}{2 \vert \varphi \vert^2} \big\langle \big( ( \partial_a \lrcorner \Omega) \cdot \partial_b - \partial_b \cdot (\partial_a \lrcorner \Omega) \big) \cdot \varphi , \varphi \big\rangle.
\end{align*}
\cite[Theorem 4.5 b)]{DeTurckKazdan} and bootstrapping finish the proof, since in the right-hand side the best regularity that can be reached is \Cka{l}.
\end{proof}

\begin{rem}
In normal coordinates $g$ is \Cka{l} by \cite[Theorem 2.1]{DeTurckKazdan}.
\end{rem}

\section{Eigenvalue pinching}

Recall that Killing spinors characterize the limit case of the Friedrich \spinc inequality. Furthermore, simply connected \spinc manifolds and spin manifolds admitting nontrivial Killing spinors are completely described in \cite{WangSimply}, \cite{WangSimply}, \cite{Baer} and \cite{Moroianu}. We will combine these two facts to obtain pinching results for Dirac eigenvalues close to the lower bound given by the Friedrich \spinc inequality \eqref{FriedrichBound}.

The pinching results will be proven by contradiction using the convergence results of Section \ref{Convergence}.

First, we start to pinch eigenvalues of the Friedrich Laplacian $\friedspinclaplace$. This result is the basis of the pinching results in this section.

\begin{prop}\label{KillFriedrich}
Let $\Lambda$, $i_0$, $d$, $K$ and $k$ be given positive real numbers, $\mu$ a given real number and $n$ a given natural number. Let $(M,g)$ be a \spinc manifold in $\precompRic$. For every $\delta > 0$ there exists a positive $\eps = \eps(n, \Lambda, i_0, d, K,  k, \mu, \delta) > 0$ such that $\lambda_k (\friedspinclaplace)  < \eps$ implies that $(M,g)$ has \Cka{1}-distance smaller than $\delta$ to a \spinc manifold with $k$ linearly independent Killing spinors with Killing number $\mu$. Furthermore, $g$ is at least \Cka{2} in harmonic coordinates.
\end{prop}

\begin{proof}
Assume the theorem to be wrong. Thus, we obtain a sequence \sequence{M_j , g_j, P_j, A} of Riemannian \spinc manifolds $(M_i, g_i)$ with principal $\sphere^1$-bundles $P_j$ with connection 1-forms $A_j$ such that we have $\lambda_k (\friedspincilaplace)~<~\eps_j$ for a vanishing sequence \sequence{\eps} and such that each $(M_j, g_j)$ is at least $\delta$ far to any \spinc manifold with $k$ linearly independent Killing spinors with Killing number $\mu$ in the \Cka{1}-topology.

By Theorem \ref{CompPrincSet} we obtain a subsequence \sequence{M, g_j, P, A} for some fixed \spinc manifold $M$ and a principal $\sphere^1$-bundle $P$ such that the metrics $g_j$ and the connection 1-forms $A_j$ converge in \Cka{1}. Since for a fixed principal $\sphere^1$-bundle on $M$ there are only finitely many equivalence classes of \spinc structures we can choose a subsequence, such that all elements have the same topological \spinc structure.

For each $j$ let $\big( \varphi_j^1, \ldots , \varphi_j^k \big)$ be an $L^2$-orthonormal family of eigenspinors for the first $k$ eigenvalues of $\friedspincilaplace$. Then for each $1 \leq l \leq k$ the sequence $\lbrace \varphi_j^l \rbrace_{j \in \bbN}$ is an $L^2$-normalized almost Killing spinor sequence converging to a Killing spinor with Killing number $\mu$ by Proposition \ref{Killconverge}. Taking appropriate subsequences, $\left\lbrace \big( \varphi_j^1 , \ldots , \varphi_j^k \big) \right\rbrace_{j \in \bbN}$ converges to $k$ linearly independent Killing spinors with Killing number $\mu$ which contradicts the assumption.

The regularity of $g$ follows from Theorem \ref{Killregularity}.
\end{proof} 

Now, we turn our attention to the special case of spin manifolds. Recall that the existence of a nontrivial Killing spinor with Killing number $\mu$ implies that the underlying spin manifold is Einsittein with Einstein constant $4\mu^2 (n-1)$. Since spin manifolds are a special case of \spinc manifolds, the same conclusion holds.

\begin{prop}\label{KillspinFriedrich}
Let $\Lambda$, $i_0$, $d$ and $k$ be given positive numbers, $mu$ a given real number and $n$ a given natural number. Then for any $\delta > 0$ there is a positive $\eps = \eps(n, \Lambda , i_0, d, k,\mu, \delta)$ such that any spin manifold in $\mfdspace$ with $\lambda_k(\friedspinclaplace ) < \eps $ has a \Cka{1}-distance less than $\delta$ to a spin manifold $(M,g)$ with $k$ linearly independent Killing spinors with Killing number $\mu$. In particular, $M$ is an Einstein manifold.
\end{prop}

\begin{rem}
For $\mu \neq 0$ and $k=1$ this result is analogous to Theorem 1.6. in \cite{Ammann} albeit with a different proof strategy.
\end{rem}

If we assume $\mu \neq 0$, then the limit manifold is compact by Myers' theorem. Furthermore, if we remove the diameter bound in $\mfdspace$, we obtain a class of manifolds which is still precompact in the pointed \Cka{1}-topology. Thus, we conjecture that for $\mu \neq 0$ the above proposition should also hold without assuming a diameter bound.

It is essential to assume $\mu \neq 0$ to obtain a compact manifold in the limit. To see this let $(\sphere^n, g)$ be the $n$-sphere with its standard metric on which we can find a nontrivial Killing spinor with Killing number $\frac{1}{2}$. Consider the sequence $( \sphere^n, j^2 \cdot g )_{j \in \bbN}$. This sequence lies in $\mathcal{M}(n,1, \pi)$. Furthermore, $(\sphere^n, j^2 \cdot g)$ admits a nontrivial Killing spinor with Killing number $\frac{1}{2j}$. Then this sequence converges to $\bbR^n$ with the standard metric in the pointed \Cka{1}-topology. In particular, $\sphere^n$ is not diffeomorphic to the limit space. 

In the case $\mu = 0$, we combine Proposition \ref{KillFriedrich} with the results for simply-connected \spinc and spin manifolds with parallel spinors  in \cite{Moroianu} and \cite{WangSimply}. 

\begin{thm}\label{PinchLaplace}
Let $n$, $\Lambda$, $i_0$, $d$, and $K$ be given positive numbers. For every $\delta > 0$ there exists a positive $\eps = \eps(n, \Lambda, i_0, d, K, \delta)$ such that any irreducible simply-connected \spinc manifold in $\precompRic$ with $\lambda_2(\spinclaplace) < \eps$ has a \Cka{1}-distance less than $\delta$ to a Ricci-flat K\"ahler spin manifold with two nontrivial parallel spinors.
\end{thm}

\begin{proof}
Let $\delta > 0$ be given. Applying Proposition \ref{KillFriedrich} with $\mu = 0$ and $k = 2$ we obtain an $\eps(n,\Lambda,i_0,d, K, \delta)$ such that any irreducible simply-connected \spinc manifold satisfying the assumptions of the theorem with this $\eps$  has \Cka{1}-distance less than $\delta$ to an irreducible simply-connected Riemannian manifold $(M, g)$ with two linearly independent parallel spinors. 

Since $g$ is at least \Cka{2} in harmonic coordinates, we can apply \cite[Theorem 3.1]{Moroianu} and obtain that $(M, g)$ is a spin manifold with two linearly independent parallel spinors. Thus, it is Ricci-flat and by the main result in \cite{WangSimply} it is also K\"ahler.
\end{proof} 

Applying the Schr\"odinger-Lichnerowicz formula, we obtain as a corollary a similar result for small Dirac eigenvalues for simply-connected \spinc manifolds with nearly non-negative scalar curvature. In addition, the theorem gives a lower bound on the second Laplace eigenvalue for compact irreducible simply-connected \spinc manifolds which are not spin. In particular, since K\"ahler manifolds are even dimensional this shows that compact irreducible simply-connected \spinc manifolds of odd dimension cannot have two arbitrary small spinorial Laplace eigenvalues.

If we consider only spin manifolds we can state a similar result for non-simply-connected spin manifolds using \cite[Theorem 1]{WangNonSimply}.

\begin{thm}
Let $n$, $\Lambda$, $i_0$, and $d$ be given positive numbers. Then for any positive $\delta$, there exists  a positive $\eps = \eps(n, \Lambda, i_0, d , \delta)$ such that any non-simply-connected locally irreducible spin manifold $(M,g)$ in $\mfdspace$ with  $\lambda_2 (\conlaplace) \leq \eps$ has a \Cka{1}-distance less than $\delta$ to some Ricci-flat K\"ahler spin manifold with two nontrivial parallel spinors.
\end{thm}

\begin{proof}
Proposition \ref{KillspinFriedrich} with $\mu = 0$ and $k = 2$ already gives us a suitable $\eps$ and shows that the limit manifold is non-simply-connected and admits two nontrivial parallel spinors. Thus, applying \cite[Theorem 1]{WangNonSimply} finishes the proof.
\end{proof} 

Again we obtain corollaries for Dirac eigenvalues and a lower bound for the second eigenvalue if the spin manifold is not diffeomorphic to a Ricci-flat K\"ahler manifold with two nontrivial parallel spinors. This is, in particular, the case for all odd dimensional manifolds.
\newline

Next, we pinch Dirac eigenvalues and consider the description of simply-connected \spinc and spin manifolds with real Killing spinors in \cite{Moroianu} and \cite{Baer}. Thus, we reformulate Proposition \ref{KillFriedrich} for Dirac eigenvalues. This can be done by using the modified Schr\"odinger-Lichnerowicz formula of Lemma \ref{modSLformula}. For brevity we set
\begin{align*}
\Iplus & \coloneqq [0 , n\mu + \eps ] \cap \sigma(D^A) ,\\
\Iminus & \coloneqq [-n\mu - \eps , 0) \cap \sigma(D^A), 
\end{align*}
for any positive numbers $\eps$ and $\mu$. In addition, we define the notion of $\mu$-Killing type.

\begin{defi}[$\mu$-Killing type]
A \spinc manifold $(M,g)$ is of $\mu$-Killing type $(p,q)$ for some positive $\mu$ if $M$ admits $p$ linearly independent Killing spinors with Killing number $\mu$ and $q$ linearly independent Killing spinors with Killing number $- \mu$.
\end{defi}

\begin{prop}\label{KillDirac}
Let $n$, $\Lambda$, $i_0$, $d$, $K$, $\mu$, $k_+$, and $k_-$ be given positive numbers. Let $(M,g)$ be a \spinc manifold in $\precompRic$. For every $\delta > 0$ there exists an $\eps = \eps (n,\Lambda, i_0, d , K , \mu , k_+ , k_ - , \delta) > 0$ such that  $\vert \Iplus \vert \geq k_+$, $\vert \Iminus \vert \geq k_-$ and $\inf_M(\Scal - c_n \vert \Omega \vert_g) \geq 4 \mu^2 n ( n-1) - \eps$ imply that $(M,g)$ has \Cka{1}-distance less than $\delta$ to some \spinc manifold of $\mu$-Killing type $(k_- , k_+)$ whose metric is at least \Cka{2} in harmonic coordinates.
\end{prop}

\begin{proof}
Assume that the theorem does not hold. Hence, we obtain, similarly to the proof of Proposition \ref{KillFriedrich}, a subsequence \sequence{M, g_j, P , A} converging in \Cka{1} such that $\inf_M(\Scal_j - c_n \vert \Omega_j \vert_{g_j}) \geq 4 \mu^2 n ( n-1) - \eps_j$ for a vanishing sequence \sequence{\eps}. In addition, $\vert I_{\eps_j, \mu}^+ \vert \geq k_+$ and $\vert I_{\eps_j, \mu}^- \vert \geq k_-$ for each $j$ and each $(M_j, g_j)$ has a \Cka{1}-distance more than $\delta$ to any \spinc manifold of $\mu$-Killing type $(k_- , k_+)$.

Let $\big(\varphi_j^1 , \ldots , \varphi_j^{k_- + k_+} \big)$ be an $L^2$-orthonormal family of corresponding eigenspinors on $(M, g_j)$. We show that each \sequence{\varphi^j} is an $L^2$-normalized almost Killing spinor sequence for $\mu$ if $\lambda_k$ is negative and $-\mu$ else. We consider only the case $\lambda_k(\Diraci) < 0$ since the other case works similarly. By using the modified Schr\"odinger-Lichnerowicz formula of Lemma \ref{modSLformula} we estimate
\begin{align*}
& \Norm{\friedi \varphi^k_j}{\Lspinci}^2 \\ 
&   \; = \langle \friedspincilaplace \varphi_j^k, \varphi_j^k \rangle_{\Lspinci} \\
& \ = \langle (\Diraci - \mu)^2 \varphi_j^k - \frac{1}{4}(\Scal_j + 2i \Omega_j) \varphi_j^k + \mu^2(n-1) \varphi_j^k, \varphi_j^k \rangle_{\Lspinci} \\
&  \;  \leq \sup_{(M,g_j)} \left( ( \lambda_k(\Diraci) - \mu )^2 - \mu^2n(n-1) + \mu^2 (n - 1) + \frac{\eps_j}{4} \right) \Norm{\varphi_j^k}{\Lspinci}^2 \\
&   \; \leq \left((n \mu + \eps_j - \mu)^2 - \mu^2(n-1)^2 + \frac{\eps_j}{4} \right) \Norm{\varphi_j^k}{\Lspinci}^2 \\
&  \; = \left( \eps_j^2 - 2\mu(n-1) \eps_j + \mu^2 (n-1)^2 - \mu^2(n-1)^2 + \frac{\eps_j}{4} \right) \Norm{\varphi_j^k}{\Lspinci}^2 \\
&  \; \leq O(\eps_j) \Norm{\varphi_j^k}{\Lspinci}^2.
\end{align*}
Thus, \sequence{\varphi^k} is an $L^2$-normalized almost Killing spinor sequence and Proposition \ref{Killconverge} applies. After we choose appropriate subsequences the family $\Big( \big(\varphi_j^1 , \ldots , \varphi_j^{k_- + k_+} \big) \Big)_{j \in \bbN}$ converges to $k_+$ linearly independent Killing spinors with Killing number $-\mu$ and $k_-$ linearly independent Killing spinors with Killing number $\mu$. This contradicts the assumption and the claim follows.
\end{proof}

\begin{prop}\label{KillspinDirac}
Let $n$, $\Lambda$, $i_0$, $d$, $\mu$, $k_+$, and $k_-$ be given positive numbers. For every $\delta$ there exists a positive $\eps = \eps(n, \Lambda, i_0, d, \mu, k_+, k_- , \delta)$ such that any spin manifold $(M,g) \in \mfdspace$ with $\vert \Iplus \vert \geq k_+$, $\vert \Iminus \vert \geq k_-$ and $\min_{(M,g)} \Scal \geq 4 \mu^2 n (n-1)  - \eps$ is at most $\delta$ far from a spin manifold of $\mu$-Killing type $(k_-, k_+)$ in the \Cka{1}-norm. \kasten
\end{prop}

Combining Proposition \ref{KillDirac} and Corollary 4.2 in \cite{Moroianu}, we obtain the following pinching result.

\begin{thm}\label{PinchDirac}
Let $n$, $\Lambda$, $i_0$, $d$, $K$, and $\mu$ be given positive numbers. Then for every $\delta > 0$ there exists an $\eps = \eps(n, \Lambda, i_0, d , K, \mu , \delta)>0$ such that any simply-connected \spinc manifold $(M,g) \in \precompRic$ with $\inf_{(M,g)} (\Scal - c_n \vert \Omega \vert_g) \geq 4\mu^2n(n-1) - \eps $ and
\begin{enumerate}\renewcommand{\labelenumi}{\roman{enumi})}
\item $\vert \Iplus \cup \Iminus \vert \geq 1$ if $n$ is even,
\item $\vert \Iplus \cup \Iminus \vert \geq 2$ if $n$ is odd,
\end{enumerate}
admits a \Cka{1}-distance at most $\delta$ to a spin manifold with one resp.\ two Killing spinors with Killing number $\pm \mu$.
\end{thm}

\begin{proof}
Fix a positive $\delta$. We apply Proposition \ref{KillDirac} with $k_+ + k_- = 1$ (resp.\ $2$) and obtain a positive $\eps$ such that the limit manifold $(M,\tilde{g})$ has 1 resp.\ 2 real Killing spinors. Then \cite[Corolary 4.2]{Moroianu} implies that this has to be a spin manifold.
\end{proof}

This theorem immediately proves the existence of a lower bound on the first, resp.\ second small Dirac eigenvalue, in the sense that it is close to the Friedrich bound, for any simply-connected \spinc manifold which is not spin.

Combining this theorem with the geometric description of spin manifolds with real Killing spinors in \cite{Baer}, we generalize Theorem 5.12 in \cite{Vargas} to simply-connected \spinc manifolds.

\begin{thm}\label{PinchSphere}
Let $n$, $\Lambda$, $i_0$, $d$, $K$ and $\mu$ be given positive numbers. Then for every $\delta > 0$ there is an $\eps = \eps(n, \Lambda, i_0, d, K, \mu , \delta) > 0$ such that any simply-connected \spinc manifold $(M,g) \in \precompRic$ with $\inf_{(M,g)} (\Scal - c_n \vert \Omega \vert_g) \geq 4\mu^2n(n-1) - \eps $ and
\begin{enumerate}\renewcommand{\labelenumi}{\roman{enumi})}
\item $\vert \Iplus \cup \Iminus \vert \geq 1$, if $n$ is even and $n \neq 6$,
\item $\vert \Iplus \vert \geq 2$ or $\vert \Iminus \vert \geq 2$, if $n=6$ or $n = 1 \pmod 4$,
\item $\vert \Iplus \vert \geq \frac{n+9}{4}$ or $\vert \Iminus \vert \geq \frac{n+9}{4}$ or $\vert \Iplus \vert \geq 1, \ \vert \Iminus \vert \geq 1$, if $n = 3 \pmod 4$,
\end{enumerate}
has \Cka{1}-distance less than $\delta$ to the sphere of constant sectional curvature $\sec = 4 \mu^2$.
\end{thm}

\begin{proof}
Fix a positive $\delta$. By Proposition \ref{KillDirac} we obtain a positive $\eps$ such that $(M,g)$ is $\delta$-close to a \spinc manifold $(M,\tilde{g})$ which is of the respective $\mu$-Killing type in the \Cka{1}-topology. 

With Corollary 4.2 of \cite{Moroianu}, it follows that $(M,\tilde{g})$ is in fact a spin manifold of the respective $\mu$-Killing type. Then by \cite[Theorems 1 to 5 ]{Baer} $(M,\tilde{g})$ has to be the sphere with $\sec~=~4\mu^2$.
\end{proof}

As before, we find a lower bound on the Dirac eigenvalues for all simply-connected \spinc manifolds not diffeomorphic to the sphere.

Recall that a spin manifold $M$ admitting at least one nontrivial real Killing spinor is positive Einstein. Thus, $M$ and its universal covering are compact by Myers' theorem. Applying this fact, we reformulate \cite[Theorem 5.4.1]{Vargas} to all spin manifolds and all Killing numbers $\mu$.

\begin{thm}
Let $n$, $\Lambda$, $i_0$, $d$, and $\mu$ be given positive numbers. Then for every $\delta > 0$ there exists an $\eps = \eps(n, \Lambda, i_0,d, \mu , \delta) > 0$ such that any spin manifold $(M,g) \in \mfdspace$ satisfying $\min_{(M,g)}(\Scal)\geq4\mu^2n(n-1)-\eps $ and
\begin{enumerate}\renewcommand{\labelenumi}{\roman{enumi})}
\item $\vert \Iplus \cup \Iminus \vert \geq 1$, if $n$ is even and $n \neq 6$,
\item $\vert \Iplus \vert \geq 2$ or $\vert \Iminus \vert \geq 2$, if $n=6$ or $n = 1 \pmod 4$,
\item $\vert \Iplus \vert \geq \frac{n+9}{4}$ or $\vert \Iminus \vert \geq \frac{n+9}{4}$ or $\vert \Iplus \vert \geq 1, \ \vert \Iminus \vert \geq 1$, if $n = 3 \pmod 4$,
\end{enumerate}
has \Cka{1}-distance less than $\delta$ to a manifold of constant sectional curvature $\sec = 4\mu^2$.
\end{thm}

\begin{proof}
Fix a positive $\delta$. Then by Corollary \ref{KillspinDirac} there is a positive $\eps$ such that any spin manifold $(M,g)$ is $\delta$-close to a spin manifold $(M,\overline{g})$ of the respective $\mu$-Killing type in the \Cka{1}-topology. Since $\overline{g}$ is a positive Einstein metric, the universal covering manifold $\tilde{M}$ is compact and has at least the same $\mu$-Killing type as $(M, \overline{g})$. Then \cite[Theorems 1 to 5]{Baer}  applied to $\tilde{M}$ finishes the proof.
\end{proof}

\section{The first nontrivial Killing number}

As another interesting application of the methods used in Proposition \ref{KillFriedrich} we obtain a lower bound on the absolute value of non-vanishing Killing numbers of spin manifolds in $\mfdspace$. This is achieved by using \cite[Theorem 3.1]{DaiWangWei} which states that if a compact simply-connected spin manifold $(M,g)$ admits a parallel spinor, there is a neighborhood $\mathcal{U}$ of $g$ in the space of smooth Riemannian metrics on $M$ such that there exists no metric of positive scalar curvature in $\mathcal{U}$.

\begin{thm}\label{Killinggap}
For given positive number $n$, $\Lambda$, $i_0$, and $d$ there is a positive $\eps = \eps (n,\Lambda, i_0, d)$ such that any spin manifold $(M,g)$ in $\mfdspace$ with a real Killing spinor with Killing number $\vert \mu \vert \leq~\eps$ has in fact a nontrivial parallel spinor.
\end{thm}

\begin{proof}
By assuming the theorem to be wrong we obtain a sequence of spin manifolds $(M_j, g_j)$ each of them admitting an $L^2$-normalized Killing spinor $\varphi_j$ with Killing number $0 < \vert \mu_j \vert \leq \eps_j$ for a vanishing sequence \sequence{\eps}. 

This sequence consists only of positive Einstein manifolds. Hence, there is a subsequence $(M, g_j)$ such that the sequence \sequence{g} converges to a metric $g$ in the $C^{\infty}$-topology by Theorem \ref{CompactSpace}.

Denote by $(\tilde{M}, \tilde{g}_j)$ the universal covering of $(M,g_j)$ for each $j$. Since for each $j$, $(M,g_j)$ is a positive Einstein manifold, $\tilde{M}$ is compact. Note that \sequence{\tilde{g}} converges in the $C^{\infty}$- topology as \sequence{g} does so. Furthermore, each $(\tilde{M} ,\tilde{g_j})$ admits an $L^2$-normalized Killing spinor with Killing number $0< \vert \mu_j \vert \leq \eps_i$.

As $\varphi_j$ is a Killing spinor with Killing number $\mu_j$ it follows that $\left. \nabla^j \right.^{\ast} \nabla^j \varphi_j = \mu_j^2 n \varphi_j$. Thus, \sequence{\varphi} is an $L^2$-normalized almost Killing spinor sequence for $\mu = 0$. Therefore, $(\tilde{M}, \tilde{g}_j)$ converges in $C^{\infty}$ to a spin manifold $(\tilde{M}, \tilde{g})$ with a nontrivial parallel spinor. 

By \cite[Theorem 3.1]{DaiWangWei} there is an open neighborhood $\mathcal{U}$ around $\tilde{g}$ in the space of smooth Riemannian metrics on $\tilde{M}$ which contains no metric of positive scalar curvature. But since \sequence{\tilde{g}} converges in the $C^{\infty}$-topology to $\tilde{g}$ there is an $J > 0$ such that $\tilde{g}_j \in \mathcal{U}$ for all $j > J$. This contradicts the assumption on $\mathcal{U}$ since the sequence \sequence{\tilde{g}} consists only of metrics of positive scalar curvature.
\end{proof}

\bibliography{literature.bib}

\newcommand{\etalchar}[1]{$^{#1}$}
\providecommand{\bysame}{\leavevmode\hbox to3em{\hrulefill}\thinspace}
\providecommand{\MR}{\relax\ifhmode\unskip\space\fi MR }
% \MRhref is called by the amsart/book/proc definition of \MR.
\providecommand{\MRhref}[2]{%
  \href{http://www.ams.org/mathscinet-getitem?mr=#1}{#2}
}
\providecommand{\href}[2]{#2}
\begin{thebibliography}{BHM{\etalchar{+}}15}

\bibitem[And90]{Anderson}
Michael~T. Anderson, \emph{Convergence and rigidity of manifolds under {R}icci
  curvature bounds}, Invent. Math. \textbf{102} (1990), no.~2, 429--445.

\bibitem[AS07]{Ammann}
Bernd Ammann and Chad Sprouse, \emph{Manifolds with small {D}irac eigenvalues
  are nilmanifolds}, Ann. Global Anal. Geom. \textbf{31} (2007), no.~4,
  409--425.

\bibitem[B{\"a}r93]{Baer}
Christian B{\"a}r, \emph{Real {K}illing spinors and holonomy}, Comm. Math.
  Phys. \textbf{154} (1993), no.~3, 509--521.

\bibitem[Bes08]{Besse}
Arthur~L. Besse, \emph{Einstein manifolds}, Classics in Mathematics,
  Springer-Verlag, Berlin, 2008, Reprint of the 1987 edition.

\bibitem[BFGK91]{BaumTwistor}
Helga Baum, Thomas Friedrich, Ralf Grunewald, and Ines Kath, \emph{Twistors and
  {K}illing spinors on {R}iemannian manifolds}, Teubner-Texte zur Mathematik
  [Teubner Texts in Mathematics], vol. 124, B. G. Teubner Verlagsgesellschaft
  mbH, Stuttgart, 1991, With German, French and Russian summaries.

\bibitem[BG92]{IdentifySpin}
Jean-Pierre Bourguignon and Paul Gauduchon, \emph{Spineurs, op\'erateurs de
  {D}irac et variations de m\'etriques}, Comm. Math. Phys. \textbf{144} (1992),
  no.~3, 581--599.

\bibitem[BHM{\etalchar{+}}15]{SpinorialApproach}
Jean-Pierre Bourguignon, Oussama Hijazi, Jean-Louis Milhorat, Andrei Moroianu,
  and Sergiu Moroianu, \emph{A spinorial approach to {R}iemannian and conformal
  geometry}, EMS Monographs in Mathematics, European Mathematical Society
  (EMS), Z\"urich, 2015. \MR{3410545}

\bibitem[Bla10]{Blair}
David~E. Blair, \emph{Riemannian geometry of contact and symplectic manifolds},
  second ed., Progress in Mathematics, vol. 203, Birkh\"auser Boston, Inc.,
  Boston, MA, 2010.

\bibitem[Bry08]{Brylinski}
Jean-Luc Brylinski, \emph{Loop spaces, characteristic classes and geometric
  quantization}, Modern Birkh\"auser Classics, Birkh\"auser Boston, Inc.,
  Boston, MA, 2008, Reprint of the 1993 edition.

\bibitem[DK81]{DeTurckKazdan}
Dennis~M. DeTurck and Jerry~L. Kazdan, \emph{Some regularity theorems in
  {R}iemannian geometry}, Ann. Sci. \'Ecole Norm. Sup. (4) \textbf{14} (1981),
  no.~3, 249--260.

\bibitem[DWW05]{DaiWangWei}
Xianzhe Dai, Xiaodong Waadmittingng, and Guofang Wei, \emph{On the stability of
  {R}iemannian manifold with parallel spinors}, Invent. Math. \textbf{161}
  (2005), no.~1, 151--176.

\bibitem[FG85]{FriedrichGrunewald}
Thomas Friedrich and Ralf Grunewald, \emph{On the first eigenvalue of the
  {D}irac operator on {$6$}-dimensional manifolds}, Ann. Global Anal. Geom.
  \textbf{3} (1985), no.~3, 265--273.

\bibitem[FK88]{FriedrichKath99}
Thomas Friedrich and Ines Kath, \emph{Vari\'et\'es riemanniennes compactes de
  dimension {$7$} admettant des spineurs de {K}illing}, C. R. Acad. Sci. Paris
  S\'er. I Math. \textbf{307} (1988), no.~19, 967--969.

\bibitem[FK89]{FriedrichKath89}
Th. Friedrich and I.~Kath, \emph{Einstein manifolds of dimension five with
  small first eigenvalue of the {D}irac operator}, J. Differential Geom.
  \textbf{29} (1989), no.~2, 263--279.

\bibitem[FK90a]{FriedrichKath90parallel}
\bysame, \emph{Compact {$5$}-dimensional {R}iemannian manifolds with parallel
  spinors}, Math. Nachr. \textbf{147} (1990), 161--165.

\bibitem[FK90b]{FriedrichKath90}
Thomas Friedrich and Ines Kath, \emph{{$7$}-dimensional compact {R}iemannian
  manifolds with {K}illing spinors}, Comm. Math. Phys. \textbf{133} (1990),
  no.~3, 543--561.

\bibitem[Fri80]{FriedrichBound}
Th. Friedrich, \emph{Der erste {E}igenwert des {D}irac-{O}perators einer
  kompakten, {R}iemannschen {M}annigfaltigkeit nichtnegativer
  {S}kalarkr\"ummung}, Math. Nachr. \textbf{97} (1980), 117--146.

\bibitem[Fri81a]{Friedrich81parallel}
\bysame, \emph{Zur {E}xistenz paralleler {S}pinorfelder \"uber {R}iemannschen
  {M}annigfaltigkeiten}, Colloq. Math. \textbf{44} (1981), no.~2, 277--290
  (1982). \MR{652586}

\bibitem[Fri81b]{Friedrich81}
Thomas Friedrich, \emph{A remark on the first eigenvalue of the {D}irac
  operator on {$4$}-dimensional manifolds}, Math. Nachr. \textbf{102} (1981),
  53--56.

\bibitem[Fri00]{Friedrich}
\bysame, \emph{Dirac operators in {R}iemannian geometry}, Graduate Studies in
  Mathematics, vol.~25, American Mathematical Society, Providence, RI, 2000,
  Translated from the 1997 German original by Andreas Nestke.

\bibitem[GGK02]{AppD}
Victor Guillemin, Viktor Ginzburg, and Yael Karshon, \emph{Moment maps,
  cobordisms, and {H}amiltonian group actions}, Mathematical Surveys and
  Monographs, vol.~98, American Mathematical Society, Providence, RI, 2002,
  Appendix J by Maxim Braverman.

\bibitem[Hat02]{Hatcher}
Allen Hatcher, \emph{Algebraic topology}, Cambridge University Press,
  Cambridge, 2002.

\bibitem[Hij86]{Hijazi}
Oussama Hijazi, \emph{Caract\'erisation de la sph\`ere par les premi\`eres
  valeurs propres de l'op\'erateur de {D}irac en dimensions {$3,$}{$4,$} {$7$}
  et {$8$}}, C. R. Acad. Sci. Paris S\'er. I Math. \textbf{303} (1986), no.~9,
  417--419.

\bibitem[Hit74]{Hitchin}
Nigel Hitchin, \emph{Harmonic spinors}, Advances in Math. \textbf{14} (1974),
  1--55.

\bibitem[HM99]{Herzlich}
Marc Herzlich and Andrei Moroianu, \emph{Generalized {K}illing spinors and
  conformal eigenvalue estimates for {${\rm Spin}^c$} manifolds}, Ann. Global
  Anal. Geom. \textbf{17} (1999), no.~4, 341--370.

\bibitem[LM89]{LawsonMichelsohn}
H.~Blaine Lawson, Jr. and Marie-Louise Michelsohn, \emph{Spin geometry},
  Princeton Mathematical Series, vol.~38, Princeton University Press,
  Princeton, NJ, 1989.

\bibitem[Mor66]{Morrey}
Charles~B. Morrey, Jr., \emph{Multiple integrals in the calculus of
  variations}, Die Grundlehren der mathematischen Wissenschaften, Band 130,
  Springer-Verlag New York, Inc., New York, 1966.

\bibitem[Mor97]{Moroianu}
Andrei Moroianu, \emph{Parallel and {K}illing spinors on {${\rm Spin}^c$}
  manifolds}, Comm. Math. Phys. \textbf{187} (1997), no.~2, 417--427.

\bibitem[PS99]{Petersen}
Peter Petersen and Chadwick Sprouse, \emph{Eigenvalue pinching for {R}iemannian
  vector bundles}, J. Reine Angew. Math. \textbf{511} (1999), 73--86.

\bibitem[{Var}07]{Vargas}
Andr\'es {Vargas Dom\'inguez}, \emph{Manifolds with {K}illing spinors and
  pinching of first {D}irac eigenvalues}, Ph.D. thesis, {R}heinische
  {F}riedrich-{W}ilhelms-{U}nivers\"at Bonn, February 2007,
  \url{http://hss.ulb.uni-bonn.de/2007/1001/1001.htm}.

\bibitem[Wan89]{WangSimply}
McKenzie~Y. Wang, \emph{Parallel spinors and parallel forms}, Ann. Global Anal.
  Geom. \textbf{7} (1989), no.~1, 59--68.

\bibitem[Wan95]{WangNonSimply}
\bysame, \emph{On non-simply connected manifolds with non-trivial parallel
  spinors}, Ann. Global Anal. Geom. \textbf{13} (1995), no.~1, 31--42.

\bibitem[Wei58]{Weil}
Andr{\'e} Weil, \emph{Introduction \`a l'\'etude des vari\'et\'es
  k\"ahl\'eriennes}, Publications de l'Institut de Math\'ematique de
  l'Universit\'e de Nancago, VI. Actualit\'es Sci. Ind. no. 1267, Hermann,
  Paris, 1958.

\end{thebibliography}
\bibliographystyle{amsalpha}

\end{document}